\documentclass{article}

\usepackage[latin2]{inputenc}
\usepackage{t1enc, anysize, enumerate, amsthm, amsmath, amssymb}
\usepackage[unicode]{hyperref}
\newtheorem{theorem}{Theorem}[section]
\newtheorem{lemma}[theorem]{Lemma}

\newtheorem{remark}[theorem]{Remark}

\newenvironment{definition}[1][Definition]{\begin{trivlist}
\item[\hskip \labelsep {\bfseries #1}]}{\end{trivlist}}

\newenvironment{question}[1][Question]{\begin{trivlist}
\item[\hskip \labelsep {\bfseries #1}]}{\end{trivlist}}

\begin{document}

\title{Union-intersecting set systems}
\author{Gyula O.H. Katona\footnote{Alfr\'ed R\'enyi Institute of Mathematics}~ and
D\'aniel T. Nagy\footnote{E\"{o}tv\"{o}s Lor\'and University, Budapest}}
\maketitle

\begin{abstract}
Three intersection theorems are proved. First, we determine the size of the largest set system, where the system of the pairwise unions is $l$-intersecting. Then we investigate set systems where the union of any $s$ sets intersect the union of any $t$ sets. The maximal size of such a set system is determined exactly if $s+t\leq 4$, and asymptotically if $s+t\geq 5$. Finally, we exactly determine the maximal size of a $k$-uniform set system that has the above described $(s,t)$-union-intersecting property, for large enough $n$.

\begin{center}
{\bf Keywords}\\ Extremal set systems, Intersecting family, Erd\H{o}s-Ko-Rado theorem, $\Delta$-system, Forbidden subposets
\end{center}
\end{abstract}

\section{Introduction}~

In the present paper we prove three intersection theorems, inspired by the following question.

\begin{question} (J. K\"{o}rner, \cite{korner})
Let $\mathcal{F}$ be a set system whose elements are subsets of $[n]=\{1,2,\dots n\}$. Assume that there are no four different sets $F_1, F_2, G_1, G_2\in \mathcal{F}$ such that $(F_1\cup F_2)\cap (G_1\cup G_2)=\emptyset$. What is the maximal possible size of $\mathcal{F}~$?
\end{question}

We will solve this problem as a special case of Theorem \ref{thma}. (See Remark \ref{kornerremark}.)

The paper is organized as follows. In the rest of this section we review some theorems that will be used later. In Section 2, the size of the largest set system will be determined, where the system of the pairwise unions is $l$-intersecting. In Section 3, we investigate set systems where the union of any $s$ sets intersect the union of any $t$ sets. The maximal size of such a set system is determined exactly if $s+t\leq 4$, and asymptotically if $s+t\geq 5$. In Section 4, we exactly determine the maximal size of a $k$-uniform set system that has the above described $(s,t)$-union-intersecting property, for large enough $n$.

The following intersection theorems will be used in the proof of Theorem \ref{thma}.

\begin{definition}
Let ${[n] \choose k}$ denote the set of all $k$-element subsets of $[n]$. A set system $\mathcal{F}$ is called $l$-intersecting, if $|A\cap B|\geq l$ holds for all $A,B\in\mathcal{F}$ ($l>0$). The following set systems (containing $k$-element subsets of $[n]$) are obviously $l$-intersecting systems.
\begin{equation}
\mathcal{F}_i=\left\{F\in {[n]\choose k} ~\Big|~ |F\cap [l+2i]|\geq l+i \right\}~~~~~~~0\leq i\leq \frac{n-l}{2}
\end{equation}
For $1\leq l\leq k\leq n$ let
\begin{equation} AK(n,k,l)=\max_{0\leq i\leq \frac{n-l}{2}} |\mathcal{F}_i|.\end{equation}
\end{definition}

It was conjectured by Frankl \cite{FR} that this is the maximum size of a $k$-uniform $l$-intersecting family. (See also \cite{FRFR}.)

\begin{theorem} {\rm (Ahlswede-Khachatrian, \cite{ak})} \label{akthm}
Let $\mathcal{F}$ be a $k$-uniform $l$-intersecting set system whose elements are subsets of $[n]$. ($1\leq l\leq k\leq n$) Then
\begin{equation}
|\mathcal{F}|\leq AK(n,k,l).
\end{equation}
\end{theorem}

\begin{theorem} {\rm (Katona, \cite{kat64}, formula (12))} \label{kat64thm}
Let $\mathcal{F}$ be a $t$-intersecting system of subsets of $[n]$. Then
\begin{equation}
|\mathcal{F}^i|+|\mathcal{F}^{n+t-1-i}|\leq {n \choose n+t-1-i}. ~~~~~~~\left(0\le i < \frac{n+t-1}{2}\right)
\end{equation}
\end{theorem}

The following results are about set systems not containing certain subposets. We will use them to prove Theorem \ref{thmb}.

\begin{definition}
Let $P$ be a finite poset with the relation $\prec$, and $\mathcal{F}$ be a family of subsets of $[n]$. We say that $P$ is contained in $\mathcal{F}$ if there is an injective mapping $f:P\rightarrow \mathcal{F}$ satisfying $a\prec b \Rightarrow f(a)\subset f(b)$ for all $a,b\in P$. $\mathcal{F}$ is called $P$-free if $P$ is not contained in it.
\end{definition}

\begin{definition}
Let $K_{xy}$ denote the poset with elements $\{a_1, a_2, \dots a_x, b_1, \dots b_y\}$, where $a_i < b_j$ for all $(i,j)$ and there is no other relation.
\end{definition}

\begin{theorem}{\rm (Katona-Tarj\'an, \cite{tarjan})} \label{tarjanthm}
Assume that $\mathcal{G}$ is a family of subsets of $[n]$ that is $K_{12}$-free and $K_{21}$-free. Then
\begin{equation}
|\mathcal{G}|\leq 2{n-1 \choose \lfloor \frac{n-1}{2} \rfloor}.
\end{equation}
\end{theorem}

\begin{theorem}{\rm (De Bonis-Katona, \cite{rfork})} \label{rfork}
Assume that $\mathcal{G}$ is a $K_{1y}$-free family of subsets of $[n]$. Then
\begin{equation}
|\mathcal{G}|\leq {n\choose \lfloor n/2 \rfloor} \left(1+\frac{2(y-1)}{n}+O(n^{-2})\right).
\end{equation}
\end{theorem}

\begin{theorem}{\rm (De Bonis-Katona, \cite{rfork})} \label{kxy}
Assume that $\mathcal{G}$ is a $K_{xy}$-free family of subsets of $[n]$. Then
\begin{equation}
|\mathcal{G}|\leq {n\choose \lfloor n/2 \rfloor} \left(2+\frac{2(x+y-3)}{n}+O(n^{-2})\right).
\end{equation}
\end{theorem}

\section{Union-l-intersecting systems}

\begin{definition}
Let $\mathcal{F}$ be a set system and $k\in\mathbb{N}$. Then $\mathcal{F}^k$ denotes the set of the $k$-element sets in $\mathcal{F}$.
\end{definition}

\begin{definition}
A set system $\mathcal{F}$ is called union-$l$-intersecting, if it satisfies $|(F_1\cup F_2)\cap (G_1\cup G_2)|\geq l$ for all sets $F_1, F_2, G_1, G_2\in \mathcal{F}$, $F_1\not=F_2$, $G_1\not=G_2$.
\end{definition}

\begin{theorem}\label{thma} Let $\mathcal{F}$ be a union-$l$-intersecting set system whose elements are subsets of $[n]$. ($n\geq 3$) Then we have the the following upper bounds for $|\mathcal{F}|$.
\begin{enumerate}[a)]
\item If $n+l$ is even, then
\begin{equation}|\mathcal{F}|\leq \sum_{i=\frac{n+l}{2}-1}^{n} {n \choose i}.\end{equation}

\item If $n+l$ is odd, then
\begin{equation}
|\mathcal{F}|\leq AK\left(n, \frac{n+l-3}{2}, l\right)+\sum_{i=\frac{n+l-1}{2}}^{n} {n \choose i}.
\end{equation}
\end{enumerate}

These are the best possible bounds.
\end{theorem}

\begin{proof}
We can assume that $\mathcal{F}$ is an {\it upset}, that is $A\in\mathcal{F}$, $A\subset B$ imply $B\in\mathcal{F}$. (If there are sets $A\subset B$, $A\in\mathcal{F}$, $B\not\in\mathcal{F}$ then we can replace $\mathcal{F}$ by $\mathcal{F}-A+B$. After finitely many steps we arrive at an upset of the same size that is still union-$l$-intersecting.)

First, assume that $l\in\{1,2\}$. Note that if $A,B\in\mathcal{F}$, $A\cap B=\emptyset$, and $|A\cup B|=n+l-3<n$, then both $A$ and $B$ can not be in $\mathcal{F}$ at the same time. $A,B \in\mathcal{F}$ and $\mathcal{F}$ being an upset would imply that there are two sets $C,D\in\mathcal{F}$ such that $A \subset C$, $B\subset D$ and $|(A\cup C) \cap (B \cup D)| = |C\cap D|= l-1$, with contradiction.

For all $0\le i < \frac{n+l-3}{2}$ define the bipartite graph $G_i(S_i, T_i, E_i)$ as follows. Let $S_i$ be the set of all the $i$-element subsets of $[n]$, let $T_i$ be the set of subsets of size $n+l-3-i$, and connect two sets $A\in S_i$ and $B\in T_i$ if they are disjoint. Then both vertex classes contain vertices of the same degree, so it follows by Hall's theorem that there is a matching that covers the smaller vertex class, that is $S_i$. Since at most one of two matched subsets can be in $\mathcal{F}$, it follows that
\begin{equation}\label{fplusf} |\mathcal{F}^i|+|\mathcal{F}^{n+l-3-i}|\leq {n \choose n+l-3-i} ~~~~~~~\left(0\le i < \frac{n+l-3}{2}\right).\end{equation}

Now let $l\geq 3$. We will show that (\ref{fplusf}) holds for all $n$.
Assume that there are $A,B\in\mathcal{F}$, $A,B\not=[n]$ such that $|A\cap B|\leq l-3$. Take $x\not\in A$ and $y\not\in B$. Then $A\cup\{x\},~B\cup\{y\}\in\mathcal{F}$, since $\mathcal{F}$ is an upset and
$$|(A\cup (A\cup\{x\}))\cap (B\cup (B\cup\{y\}))|
=|(A\cup\{x\})\cap(B\cup\{y\})|\leq |(A\cap B) \cup \{x,y\}|=l-3+2 < l.$$
So $\mathcal{F}-\{[n]\}$ is an $(l-2)$-intersecting system, so $\mathcal{F}$ is one too.
Now use Theorem \ref{kat64thm} with $t=l-2$. ($l-2$ is positive since $l\geq 3$.) It gives us that (\ref{fplusf}) holds for all $n$ and $l$.

Assume that $l$ is a positive integer $n+l$ is odd. Then the sets in $\mathcal{F}^{\frac{n+l-3}{2}}$ form an $l$-intersecting family. $A,B \in\mathcal{F}^{\frac{n+l-3}{2}}$, $|A\cap B|\le l-1$ and $\mathcal{F}$ being an upset would imply that there are two sets $C,D\in\mathcal{F}$ such that $A \subset C$, $B\subset D$ and $|(A\cup C) \cap (B \cup D)| = |C\cap D|= l-1$. So Theorem \ref{akthm} provides an upper bound:
\begin{equation}\label{ahls} |\mathcal{F}^{\frac{n+l-3}{2}}|\le AK\left(n, \frac{n+l-3}{2}, l\right).\end{equation}

The upper bounds of the theorem follow after some calculations. When $l\leq 2$, the inequalities of (\ref{fplusf}) imply
\begin{equation}
|\mathcal{F}|=\sum_{i=0}^n |\mathcal{F}^i|=|\mathcal{F}^{\frac{n+l-3}{2}}|+\sum_{i=0}^{\lfloor\frac{n+l}{2}-2\rfloor} (|\mathcal{F}^i|+|\mathcal{F}^{n+l-3-i}|)+\sum_{i=n+l-2}^n |\mathcal{F}^i| \leq |\mathcal{F}^{\frac{n+l-3}{2}}|+\sum_{i=\lceil\frac{n+l}{2}-1\rceil}^n {n \choose i}.
\end{equation}
Let $l\geq 3$. Since $\mathcal{F}$ is $(l-2)$-intersecting, $|\mathcal{F}^i|=0$ for all $i<l-2$. Using (\ref{fplusf}), we get
\begin{equation}
|\mathcal{F}|=\sum_{i=l-3}^n |\mathcal{F}^i|= |\mathcal{F}^{\frac{n+l-3}{2}}|+\sum_{i=l-3}^{\lfloor\frac{n+l}{2}-2\rfloor} (|\mathcal{F}^i|+|\mathcal{F}^{n+l-3-i}|)\leq |\mathcal{F}^{\frac{n+l-3}{2}}|+\sum_{i=\lceil\frac{n+l}{2}-1\rceil}^n {n \choose i}.
\end{equation}

We got the same inequality in the two cases. The upper bounds of the theorem follow immediately, since $\mathcal{F}^{\frac{n+l-3}{2}}=\emptyset$, if $n+l$ is even, and $|\mathcal{F}^{\frac{n+l-3}{2}}|\le AK\left(n, \frac{n+l-3}{2}, l\right)$, if $n+l$ is odd (see (\ref{ahls})).

To verify that the given bounds are best possible, consider the following union-$l$-intersecting set systems. When $n+l$ is even, take all the subsets of size at least $\frac{n+l}{2}-1$. When $n+l$ is odd, take all the subsets of size at least $\frac{n+l-1}{2}$ and an $\frac{n+l-3}{2}$-uniform $l$-intersecting set system of size $AK\left(n, \frac{n+l-3}{2}, l\right)$.
\end{proof}

\begin{remark} \label{kornerremark}
Let us formulate the special case $l=1$ what was originally asked by K\"{o}rner. The Erd\H{o}s-Ko-Rado theorem \cite{ekr} states that the size of the largest $k$-uniform intersecting system of subsets of $[n]$ is ${n-1\choose k-1}$, when $n\ge 2k$. It means that $AK(n,\frac{n}{2}-1,1)={n-1\choose \frac{n}{2}-2}$, so the best upper bound for $|\mathcal{F}|$ when $l=1$ is
\begin{equation} |\mathcal{F}|\leq \begin{cases} \displaystyle\sum_{i=(n-1)/2}^n {n \choose i} & \mbox{if } n\mbox{ is odd,} \\ {n-1 \choose \frac{n}{2}-2}+\displaystyle\sum_{i=\frac{n}{2}}^n {n \choose i} & \mbox{if } n\mbox{ is even.} \end{cases} \end{equation}
\end{remark}

\section{Considering the union of more subsets}~

In this section we investigate a variation of the problem where we take the union of $s$ and $t$ subsets instead of 2 and 2.

\begin{definition}
A set system $\mathcal{F}$ is called $(s,t)$-union-intersecting if it has the property that for all $s+t$ pairwise different sets $F_1, F_2, \dots F_s, G_1, \dots G_t\in\mathcal{F}$
\begin{equation}
\left(\bigcup_{i=1}^s F_i\right) \cap \left(\bigcup_{j=1}^t G_j\right) \not= \emptyset.
\end{equation}
The size of the largest $(s,t)$-union-intersecting system whose elements are subsets of $[n]$ is denoted by $f(n,s,t)$.
\end{definition}

In this section we determine the value of $f(n,s,t)$ exactly when $s+t\leq 4$ and asymptotically in the other cases. Since $f(n,s,t)=f(n,t,s)$, we can assume that $s\leq t$.

\begin{theorem}\label{thmb} Let $n\ge 3$.
\begin{enumerate}[a)]
\item \begin{equation} f(n,1,1)=2^{n-1}. \end{equation}
\item \begin{equation} f(n,1,2) = \begin{cases} \displaystyle\sum_{i=n/2}^n {n \choose i} & \mbox{if } n\mbox{ is even,} \\ {n-1 \choose \frac{n-3}{2}}+\displaystyle\sum_{i=(n+1)/2}^n {n \choose i} & \mbox{if } n\mbox{ is odd.} \end{cases} \end{equation}
\item \begin{equation} f(n,2,2) = \begin{cases} \displaystyle\sum_{i=(n-1)/2}^n {n \choose i} & \mbox{if } n\mbox{ is odd,} \\ {n-1 \choose \frac{n}{2}-2}+\displaystyle\sum_{i=n/2}^n {n \choose i} & \mbox{if } n\mbox{ is even.} \end{cases} \end{equation}
\item \begin{equation} f(n,1,3) = \begin{cases} \displaystyle\sum_{i=n/2}^n {n \choose i} & \mbox{if } n\mbox{ is even,} \\ {n-1 \choose \frac{n-1}{2}}+\displaystyle\sum_{i=(n+1)/2}^n {n \choose i} & \mbox{if } n\mbox{ is odd.} \end{cases} \end{equation}
\item If $t\ge 4$, then
\begin{equation}2^{n-1}+\frac{1}{2}{n \choose \lfloor n/2\rfloor} \leq f(n,1,t) \leq 2^{n-1}+{n \choose \lfloor n/2\rfloor}\left(\frac{1}{2}+\frac{t-2}{n}+O(n^{-2})\right). \end{equation}
\item If $s\ge 2$ and $t\ge 3$, then
\begin{equation}2^{n-1}+{n \choose \lfloor n/2\rfloor}\frac{n}{n+2} \leq f(n,s,t) \leq 2^{n-1}+{n \choose \lfloor n/2\rfloor}\left(1+\frac{t+s-3}{n}+O(n^{-2})\right). \end{equation}
\end{enumerate}
\end{theorem}

\begin{proof}
\begin{enumerate}[a)]
\item (See \cite{ekr}.) $f(n,1,1)$ is the size of the largest intersecting system among the subsets of $[n]$. It is at most $2^{n-1}$, since a subset and its complement cannot be in the intersecting system at the same time. By choosing all the subsets containing a fixed element, we get an intersecting system of size $2^{n-1}$.

\item Let $\mathcal{F}$ be a $(1,2)$-union-intersecting system. We can assume that $\mathcal{F}$ is an upset. Note that if $A,B\in\mathcal{F}$, $A\cap B=\emptyset$, and $|A\cup B|=n-1$, then $A$ and $B$ can not be in $\mathcal{F}$ at the same time. To see this, take let $\{x\}=[n]-(A\cup B)$. Then $A\cap(B\cup (B\cup \{x\}))=A\cap(B\cup \{x\})=\emptyset$.

For all $0\le i < \frac{n-1}{2}$ define the bipartite graph $G_i(S_i, T_i, E_i)$ as follows. Let $S_i$ be the set of all the $i$-element subsets of $[n]$, let $T_i$ be the set of subsets of size $n-1-i$, and connect two sets $A\in S_i$ and $B\in T_i$ if they are disjoint. Then both vertex classes contain vertices of the same degree, so it follows by Hall's theorem that there is a matching that covers the smaller vertex class, that is $S_i$. Since at most one of two matched subsets can be in $\mathcal{F}$, it follows that
\begin{equation}\label{fplusf2} |\mathcal{F}^i|+|\mathcal{F}^{n-1-i}|\leq {n \choose n-1-i} ~~~~~~~\left(0\le i < \frac{n-1}{2}\right).\end{equation}
When $n$ is even, these inequalities together imply
\begin{equation} |\mathcal{F}|\leq \sum_{i=n/2}^{n} {n \choose i}.\end{equation}

Assume that $n$ is odd. Then $\mathcal{F}^{\frac{n-1}{2}}$ is an intersecting family. $A,B \in\mathcal{F}^{\frac{n-1}{2}}$, $A\cap B= \emptyset$ and $\mathcal{F}$ being an upset would imply that there is a set $C\in\mathcal{F}$ such that $B\subset C$ and $|A\cap (B \cup C)| = |A\cap C|= \emptyset$. So the Erd\H{o}s-Ko-Rado theorem provides an upper bound:
\begin{equation} |\mathcal{F}^{\frac{n-1}{2}}|\le {n-1 \choose \frac{n-3}{2}}.\end{equation}
This, together with the inequalities of (\ref{fplusf2}) implies
\begin{equation} |\mathcal{F}| \leq {n-1 \choose \frac{n-3}{2}}+\sum_{i=(n+1)/2}^n {n \choose i}.\end{equation}

To verify that the given bounds are best possible, consider the following (1,2)-union-intersecting set systems. When $n$ is even, take all the subsets of size at least $\frac{n}{2}$. When $n$ is odd, take all the subsets of size at least $\frac{n+1}{2}$ and the subsets of size $\frac{n-1}{2}$ containing a fixed element.

\item We already solved this problem as the case $l=1$ in Theorem \ref{thma}. (See Remark \ref{kornerremark}.)

\item Let $\mathcal{F}$ be a $(1,3)$-union-intersecting system of subsets of $[n]$. Let $\mathcal{F}'=\{[n]-F~\big|~F\in\mathcal{F}\}$, and let $\mathcal{G}=\mathcal{F}\cap\mathcal{F}'$. Now we prove that $\mathcal{G}$ is $K_{1,2}$-free and $K_{2,1}$-free. (See Section 1 for the definitions.) Since $\mathcal{G}$ is invariant to taking complements, it is enough to show that $\mathcal{G}$ is $K_{12}$-free. Assume that there are three pairwise different sets $A,B,C\in\mathcal{G}$, such that $A\subset B$ and $A\subset C$. Then the sets $A,[n]-A, [n]-B, [n]-C\in\mathcal{F}$ would satisfy
    \begin{equation}A\cap (([n]-A)\cup ([n]-B)\cup ([n]-C))=A\cap([n]-A)=\emptyset. \end{equation}

Theorem \ref{tarjanthm} gives us the following upper bound for a set system that is $K_{12}$-free and $K_{21}$-free:
\begin{equation}
|\mathcal{G}|\leq 2{n-1 \choose \lfloor \frac{n-1}{2} \rfloor}.
\end{equation}

Since $2|\mathcal{F}|\leq 2^n+|\mathcal{G}|$, we have
\begin{equation}
|\mathcal{F}|\leq 2^{n-1}+{n-1 \choose \lfloor \frac{n-1}{2} \rfloor}=\begin{cases} \displaystyle\sum_{i=n/2}^n {n \choose i} & \mbox{if } n\mbox{ is even,} \\ {n-1 \choose \frac{n-1}{2}}+\displaystyle\sum_{i=(n+1)/2}^n {n \choose i} & \mbox{if } n\mbox{ is odd.} \end{cases}
\end{equation}

To verify that the given bounds are best possible, consider the following (1,3)-union-intersecting set systems. When $n$ is even, take all the subsets of size at least $\frac{n}{2}$. When $n$ is odd, take all the subsets of size at least $\frac{n+1}{2}$ and the subsets of size $\frac{n-1}{2}$ {\it not} containing a fixed element.

\item Let $\mathcal{F}$ be a $(1,t)$-union-intersecting system of subsets of $[n]$. Define $\mathcal{G}$ as above, and note that $\mathcal{G}$ is $K_{1,t-1}$-free. Theorem \ref{rfork} implies
\begin{equation}
|\mathcal{G}|\leq {n\choose \lfloor n/2 \rfloor}
\left(1+\frac{2(t-2)}{n}+O(n^{-2})\right).
\end{equation}
Since $2|\mathcal{F}|\leq 2^n+|\mathcal{G}|$, we have
\begin{equation}
|\mathcal{F}|\leq 2^{n-1}+{n \choose \lfloor n/2\rfloor}\left(\frac{1}{2}+\frac{t-2}{n}+O(n^{-2})\right).
\end{equation}
The lower bound follows obviously from
\begin{equation}
f(n,1,t)\geq f(n,1,3)=2^{n-1}+{n-1 \choose \lfloor \frac{n-1}{2} \rfloor} \geq 2^{n-1}+\frac{1}{2}{n \choose \lfloor n/2\rfloor}.
\end{equation}

\item Let $\mathcal{F}$ be an $(s,t)$-union-intersecting system of subsets of $[n]$. Define $\mathcal{G}$ as above. Now we prove that $\mathcal{G}$ is $K_{s,t}$-free. Assume that $A_1, A_2, \dots A_s, B_1, \dots B_t\in\mathcal{G}$ are pairwise different subsets and $A_i\subset B_j$ for all $(i,j)$. Then $\left(\displaystyle\bigcup_{i=1}^s A_i\right) \cap \left(\displaystyle\bigcup_{j=1}^t ([n]-B_j)\right) = \emptyset$. It is a contradiction, since $[n]-B_j\in\mathcal{F}$ for all $j$.

    Theorem \ref{kxy} implies
\begin{equation}
|\mathcal{G}|\leq {n\choose \lfloor n/2 \rfloor} \left(2+\frac{2(s+t-3)}{n}+O(n^{-2})\right).
\end{equation}
Since $2|\mathcal{F}|\leq 2^n+|\mathcal{G}|$, we have
\begin{equation}
|\mathcal{F}|\leq 2^{n-1}+{n \choose \lfloor n/2\rfloor}\left(1+\frac{t+s-3}{n}+O(n^{-2})\right).
\end{equation}

The lower bound follows from
\begin{equation}
f(n,s,t)\geq f(n,2,2) \geq 2^{n-1}+{n \choose \lfloor n/2\rfloor}\frac{n}{n+2}.
\end{equation}
(The second inequaility can be verified by elementary calculations since the exact value of $f(n,2,2)$ is known. Equality holds when $n$ is even.)

\end{enumerate}
\end{proof}

\section{The k-uniform case}~

In this section we determine the size of the largest $k$-uniform $(s,t)$-union-intersecting set system of subsets of $[n]$ for all large enough $n$.

\begin{theorem}\label{uniformthm}
Assume that $1\leq s \leq t$ and $\mathcal{F}\subset {[n] \choose k}$ is an $(s,t)$-union-intersecting set system. Then
$$|\mathcal{F}|\leq {n-1 \choose k-1}+s-1$$
holds for all $n>n(k,t)$.
\end{theorem}

\begin{remark}
There is $k$-uniform $(s,t)$-union-intersecting set system of size ${n-1 \choose k-1}+s-1$. Take all $k$-element sets containing a fixed element, then take $s-1$ arbitrary sets of size $k$.
\end{remark}

We need some preparation before we can start the proof of Theorem \ref{uniformthm}.

\begin{definition}
A sunflower (or $\Delta$-system) with $r$ petals and center $M$ is a family $\{S_1, S_2,\dots S_r\}$ where $S_i\cap S_j=M$ for all $1\leq i<j\leq r$.
\end{definition}

\begin{lemma} {\rm (Erd\H{o}s-Rado \cite{sunflower})} \label{sunlemma}
Assume that $\mathcal{A}\subset {[n] \choose k}$ and $|\mathcal{A}|>k!(r-1)^k$. Then $\mathcal{A}$ contains a sunflower with $r$ petals as a subfamily.
\end{lemma}

\begin{proof}
We prove the lemma by induction on $k$. The statement of the lemma is obviously true when $k=0$. Let $\mathcal{A}\subset {[n] \choose k}$ a set system not containing a sunflower with $r$ petals. Let $\{A_1, A_2, \dots A_m\}$ be a maximal family of pairwise disjoint sets in $\mathcal{A}$. Since pairwise disjoint sets form a sunflower, $m\leq r-1$. For every $x\in\displaystyle\bigcup_{i=1}^m A_i$, let $\mathcal{A}_x=\{S-\{x\}~\big|~ S\in\mathcal{A},~ x\in S\}$. Then each $\mathcal{A}_x$ is a $k-1$-uniform set system not containing sunflowers with $r$ petals. By induction we have $|\mathcal{A}_x|\leq (k-1)!(r-1)^{k-1}$. Then
\begin{equation}
|\mathcal{A}|\leq (k-1)!(r-1)^{k-1}\left|\bigcup_{i=1}^m A_i \right|\leq (k-1)!(r-1)^{k-1}\cdot k(r-1)=k!(r-1)^k.
\end{equation}
\end{proof}

\begin{lemma}\label{uniformla}
Let $c$ be a fixed positive integer. If $\mathcal{A}\subset {[n] \choose k}$, $K\subset [n]$, $|K|\leq c$ and $|A\cap K|\geq 2$ holds for every $A\in\mathcal{A}$, then $|\mathcal{A}|\leq O(n^{k-2})$.
\end{lemma}

\begin{proof}
Choose a set $K\subset K'$ with $|K'|=c$. The conditions of the lemma are also satisfied with $K'$.
\begin{equation}
|\mathcal{A}|\leq \sum_{i=2}^c {c \choose i}{n-c \choose k-i}.
\end{equation}
The right hand side is a polynomial of $n$ with degree $k-2$, so $|\mathcal{A}|\leq O(n^{k-2})$.
\end{proof}

\begin{lemma}\label{uniformlb}
Let $s,t$ be fixed positive integers. Let $\mathcal{A}, \mathcal{B}\subset {[n] \choose k}$. Assume that $|\mathcal{B}|\geq s$, $\mathcal{A}\cup\mathcal{B}$ is $(s,t)$-union-intersecting, and there is an element $a$ such that $a\in A$ holds for every $A\in\mathcal{A}$, and $a\not\in\mathcal{B}$ holds for every $B\in\mathcal{B}$. Then $|\mathcal{A}|\leq O(n^{k-2})$.
\end{lemma}

\begin{proof}
Choose $s$ different sets  $B_1, B_2, \dots B_s\in\mathcal{B}$. Let $\mathcal{A}'= \{ A\in\mathcal{A} ~\big|~ A\cap\displaystyle\bigcup_{i=1}^s B_i\not=\emptyset\}$. Since $\mathcal{A}\cup\mathcal{B}$ is $(s,t)$-union-intersecting, $|\mathcal{A}-\mathcal{A}'|\leq t-1$.
\begin{equation} \left|A\cap\left(\{a\}\cup \bigcup_{i=1}^s B_i\right)\right|\geq 2 \end{equation}
holds for all $A\in\mathcal{A'}$. Use Lemma \ref{uniformla} with $K=\{a\}\cup \displaystyle\bigcup_{i=1}^s B_i$ and $c=sk+1$, it implies $|\mathcal{A}'|\leq O(n^{k-2})$. Then $|\mathcal{A}|\leq O(n^{k-2})+(t-1)=O(n^{k-2})$.
\end{proof}

\begin{proof} (of Theorem \ref{uniformthm})
Use Lemma \ref{sunlemma} with $r=ks+t$. If $|\mathcal{F}|>k!(ks+t)^k$, then $\mathcal{F}$ contains a sunflower $\{S_1, S_2, \dots S_{ks+t}\}$ as a subfamily. (Note that ${n-1\choose k-1}>k!(ks+t)^k$ holds for large enough $n$.) Let $M$ denote the center of the sunflower and introduce the notations $|M|=\{a_1, a_2, \dots a_m\}$ and $C_i=S_i-M~~(1\leq i\leq ks+t)$. Let $\mathcal{F}_0=\{F\in\mathcal{F}~\big|~F\cap M=\emptyset\}$, and $\mathcal{F}_i=\{F\in\mathcal{F}~\big|~a_i\in F\}$ for $1\leq i\leq m$.

Assume that $|\mathcal{F}_0|\geq s$. Let $B_1, B_2, \dots B_s\in\mathcal{F}_0$ be different sets. Since $\left|\displaystyle\bigcup_{i=1}^s B_i\right|\leq ks$, and the sets $\{C_1, C_2, \dots C_{ks+t}\}$ are pairwise disjoint, there some indices $i_1, i_2, \dots i_t$ such that
\begin{equation}
\emptyset=\left(\bigcup_{i=1}^s B_i\right) \cap \left(\bigcup_{j=1}^t C_{i_j}\right) = \left(\bigcup_{i=1}^s B_i\right) \cap \left(\bigcup_{j=1}^t S_{i_j}\right).
\end{equation}
It contradicts our assumption that $\mathcal{F}$ is $(s,t)$-union-intersecting, so $|\mathcal{F}_0|\leq s-1$.

Note that the obvious inequality $|\mathcal{F}_i|\leq {n-1\choose k-1}$ holds for all $1\leq i\leq m$, so the statement of the theorem is true if $|\mathcal{F}-\mathcal{F}_i|\leq s-1$ holds for any $i$.

Finally, assume that $|\mathcal{F}-\mathcal{F}_i|\geq s$ holds for all $1\leq i\leq m$. Using Lemma \ref{uniformlb} with $\mathcal{A}=\mathcal{F}_i$ and $\mathcal{B}=\mathcal{F}-\mathcal{F}_i$, we get that $|\mathcal{F}_i|=O(n^{k-2})$. Then
\begin{equation}
|\mathcal{F}|\leq \sum_{i=0}^m |\mathcal{F}_i|\leq (s-1)+m\cdot O(n^{k-2})= O(n^{k-2}).
\end{equation}
Since ${n-1 \choose k-1}+s-1$ is a polynomial of $n$ with degree $k-1$, $|\mathcal{F}|\leq {n-1 \choose k-1}+s-1$ holds for large enough $n$.
\end{proof}

\begin{remark}
Note that Theorem \ref{uniformthm} generalizes the Erd\H{o}s-Ko-Rado theorem \cite{ekr} for large enough $n$, since letting $s=1$ and $t\geq 2$, we get the same upper bound for $|\mathcal{F}|$ while having weaker conditions on $\mathcal{F}$.
\end{remark}

\begin{question}
Let $\mathcal{F}$ be a set system satisfying the conditions of Theorem \ref{uniformthm}. What is the best upper bound for $|\mathcal{F}|$, when $n$ is small?
\end{question}

\end{document}